\documentclass{article}
\usepackage[utf8]{inputenc}
\usepackage{graphicx}
\usepackage{amsthm}
\usepackage{amsmath}
\usepackage{amssymb}
\usepackage{authblk}
\usepackage{xcolor}
\usepackage{verbatim}
\usepackage{subcaption}
\usepackage{amsthm}
\usepackage{tikz}
\usepackage{xfrac}
\usepackage{float}

\usetikzlibrary{decorations.pathreplacing}

\newcommand{\R}{{\mathbb R}}
\newcommand{\EE}{{\mathbb E}}
\newcommand{\PP}{{\mathbb P}}
\newcommand{\N}{{\mathbb N}}

\newcommand{\cX}{{\cal{X}}}
\newcommand{\cN}{{\cal{N}}}

\newcommand{\uu}{{\mathbf{u}}}

\newcommand{\1}{{\mathbf{1}}}

\newtheorem{theorem}{Theorem}[section]
\newtheorem{proposition}[theorem]{Proposition}
\newtheorem{corollary}{Corollary}[theorem]

\usepackage[margin=1in]{geometry}
\counterwithin{equation}{section}
\begin{document}
\begin{center}
{\Large\bf A non-homogeneous Semi-Markov model for
Interval Censoring}\\[0.2in]
{\large\sc M.N.M.\ van Lieshout} and {\large\sc R.L.\ Markwitz}\\[0.2in]
{\footnotesize\em
Centrum Wiskunde \& Informatica (CWI) \\
P.O.~Box 94079, NL-1090 GB Amsterdam, The Netherlands}\\[0.1in]
{\footnotesize\em
Department of Applied Mathematics, 
University of Twente \newline
P.O.~Box 217, NL-7500 AE, Enschede, The Netherlands}
\end{center}

\begin{verse}
{\footnotesize
\noindent
Previous approaches to modelling interval-censored 
data have often relied on assumptions of homogeneity in the
sense that the censoring mechanism, the underlying distribution 
of occurrence times, or both, are assumed to be time-invariant. 
In this work, we introduce a model which allows for non-homogeneous
behaviour in both cases. In particular, we outline a censoring 
mechanism based on semi-Markov processes in which interval 
generation is assumed to be time-dependent and we propose a 
Markov point process model for the underlying occurrence time 
distribution. We prove the existence of this process and derive 
the conditional distribution of the occurrence times given the 
intervals. We provide a framework within which the process can 
be accurately modelled, and subsequently compare our model to 
homogeneous approaches by way of a parametric example.\\[0.2in]

\noindent
{\em AMS Mathematics Subject Classification (2020 Revision):}
60G55, 60K15.

\noindent
{\em Key words \& Phrases:} inhomogeneity, interval-censoring, 
marked temporal 
point process, Markov point process, semi-Markov process.}
\end{verse}

\section{Introduction}
\label{sec:intro}

In previous work \cite{LiesMark23}, we developed statistical 
methods for state estimation on interval-censored data. The 
motivating example was that of determining the occurrence times 
of residential burglaries based on police reports. In the 
criminology literature, such data are known as aoristic crime 
data \cite{Ratcliffe02, Ratcliffe98}. Aoristic crime studies 
have mainly focused on ad hoc methods \cite{Ashby13}, which can
be helpful but may miss dependencies such as the near-repeat 
effect \cite{Bernasco09}. We developed a Bayesian statistical 
method that can account for inter-event dependencies 
\cite{LiesMark23}.

Our approach assumed that for each event occurrence the
censoring mechanism is governed by a stochastic process.
Specifically, an alternating renewal process was used to 
split time up into observable and partially observable periods 
according to the two phases of the renewal process. Either 
the event is fully observed, in which case the exact time of 
occurrence is recorded, or only the interval between two 
jumps is recorded. The approach was shown to lead to a 
tractable mark distribution and is therefore amenable to 
Monte Carlo methods for simulation. 

The censoring mechanism based on alternating renewal processes 
imposes time-homogeneity. In reality, events rarely occur 
homogeneously in time. For instance, returning to the motivating 
example, there may be times of day that are more likely to be 
censored due to the periodic behaviour of potential victims, 
such as being at work or asleep. Additionally, burglars may 
choose to commit crimes at different rates at certain times of 
the day based on their perception of victim behaviour. Thus, 
there may be inhomogeneity in both the underlying distribution 
of occurrence times and the censoring mechanism. 

This paper introduces a new model that rectifies these
shortcomings. For the censoring mechanism, we propose a
non-homogeneous semi-Markov process \cite{Janssen13,Janssen06,KoroSwis95,Ross96}. Conditional 
intensity-based methods \cite{Haezendonck80} are used
to guarantee existence and we derive the joint, marginal 
and conditional distributions of starting point 
and length for each occurrence time. We then propose a marked 
point process model \cite{Daley03} for the complete data
using a non-homogeneous Markov point process \cite{Lies00} 
for the ground process of event occurrences and a mark kernel 
based on the semi-Markov process. We 
illustrate the model by means of parametric examples that 
can describe various types of non-homogeneous behaviour, 
culminating in a comparison of non-homogeneous and homogeneous 
models.

The plan of this paper is as follows. 
In Section~\ref{sec:semimark-sub} we recall the definition
of a semi-Markov process on the half line and give an 
explicit expression for the joint distribution of age and
excess. In Section~\ref{sec:complete_model} we formulate
our marked point process model and study the conditional
distribution of the ground process given the union of marks. 
In Section~\ref{sec:example} we present some parametric
examples; a demonstration of the model in action is given in Section~\ref{sec:simulation}.

\section{The non-homogeneous semi-Markov process}
\label{sec:semimark-sub}

\subsection{Definition and notation}
\label{sec:defs}

Let $(\Omega, \mathcal{A}, {P})$ be a probability space.
Consider the two-dimensional stochastic process $(S_i, X_i)$, 
$i \in \mathbb{N}_0$, on $(\Omega, \mathcal{A}, {P})$ 
with values in $\{ 0, 1 \} \times \mathbb{R}^+$. Here, $S_i$
denotes the $i$-th state that the process is in and 
$0 = X_0 \leq X_1 \leq \dots$ are the jump times.
Call a time interval that the process spends in state $0$ 
a $Z$-phase and state $1$ a $Y$-phase, in 
analogy to \cite{LiesMark23}. We set $S_0 = 1$, as is convention.

The tuple $(S_n, X_n)_{n=1}^{\infty}$ defines a 
\textit{non-homogeneous semi-Markov process} if
\begin{align}
    \mathbb{P}(S_{n+1} = j&, X_{n+1} \leq x\,|\, (S_0, X_0) = 
    (s_0, x_0), \dots, (S_n, X_n) = (s_n, x_n)) \nonumber \\
    &= \mathbb{P}(S_{n+1} = j, X_{n+1} \leq x\,|\, (S_n, X_n) = (s_n, x_n)) \nonumber \\
    &= \mathbb{P}(S_{n+1} = j, X_{n+1} - X_n \leq x - x_n \,|\, (S_n, X_n) = (s_n, x_n)) \nonumber \\
    &= \mathbb{P}(S_{1} = j, X_{1} - X_0 \leq x - x_0 \,|\, (S_0, X_0) = (s_0, x_0)),
    \label{eq:markov_renewal}
\end{align}
i.e.\ the joint conditional probability of the sojourn time
$T_{n+1} = X_{n+1} - X_{n}$ in the $n$-th state
and the next state $S_{n+1}$ depends only on the $n$-th state
$S_n$ and its jump time $X_n$, not on the entire history
of the process \cite{Cinlar69, Janssen06, KoroSwis95, Ross96}
nor on the index $n$. 
This process is 
only Markov at the jump times, hence the name 
\textit{semi-Markov}.

It follows from (\ref{eq:markov_renewal}) that
the distribution of a non-homogeneous semi-Markov 
process is completely specified by the \textit{starting state} 
(or its probability distribution) and a
\textit{semi-Markov kernel} $G$ that describes the 
transition rates from state $i$ to state $j$. Formally, 
for $\tau \geq 0$, $x\geq 0$ and $i,j \in \{ 0,1\}$,
\begin{equation}
    G_{ij}(x, \tau) = \mathbb{P}(S_{n+1} = j,\,T_{n+1} \leq \tau \,|\,S_{n} = i, X_{n} = x),
    \label{eq:semi_kernel}
\end{equation}
regardless of $n = 0, 1, \dots$. As the process alternates, we can write $G_{10}(x, \tau) = 
G_{Y}(x, \tau)$ and $G_{01}(x, \tau) = G_{Z}(x, \tau)$, 
the subscript denoting the state that the process is in 
after jump time $x$. 
In the remainder of this paper, we shall assume that,
for all $x\geq 0$, $G_Y(x, \cdot)$ and $G_Z(x, \cdot)$ are 
absolutely continuous with respect to Lebesgue measure and 
write $g_Y(x,\cdot)$ and $g_Z(x, \cdot)$ respectively for
their Radon--Nikodym derivatives.

\subsection{Conditional intensities and non-explosion conditions}
\label{sec:hazard_rates}

The \textit{conditional intensity}, also known as the 
stochastic intensity, of a temporal point process describes 
the infinitesimal conditional probability of occurrence given 
the history of the process \cite{Karr91}. More precisely,
for $n=0, 1, \dots$ and $0 = x_0 \leq x_1 \leq \dots \leq x_n
\leq x$,
\begin{equation}
 \lambda_{n+1}(x; x_1, \dots, x_n) \, dx = 
 \mathbb{P}(X_{n+1} \leq x+dx\,|\,
 X_{n+1} \geq x, X_0 = 0, X_1 = x_1, \dots, X_n = x_n).
    \label{eq:cond_intensity_x}
\end{equation}
For a non-homogeneous semi-Markov process, the 
$\lambda_{n+1}(\cdot; \cdot)$ are closely related to the 
hazard rates of the sojourn times. To see this, recall 
that $S_0 = 1$ and assume that $n+1$ is odd. Then the 
conditional intensity of the jump process at time 
$x$ given jumps at times 
$0 \leq x_1 \leq x_2 \leq \cdots \leq x_n$ can be 
simplified as 
\begin{align}
    \lambda_{n+1}(x; x_1, \dots, x_n) dx &= 
    \mathbb{P}(X_{n+1} \leq x+ dx\,|\, X_{n+1} \geq x, 
    X_n = x_n, S_n = 1) \nonumber \\
    &=\frac{\mathbb{P}(x - x_n \leq T_{n+1} 
    \leq x - x_n + dx\,|\,X_n = x_n, S_n = 1)}{
    \mathbb{P}(T_{n+1} \geq x-x_n\,|\, X_n = x_n, S_n = 1)}
    \nonumber \\
    &=\frac{g_Y(x_n, x - x_n) \, dx }{1-G_Y(x_n, x - x_n)}
    \label{eq:hazard_rate_semimark_derivation}
\end{align}
whenever well-defined and using absolute continuity of
the semi-Markov kernel. When $G_Y(x_n, x-x_n) = 1$,
the conditional intensity is set to zero. For even $n+1$,
a similar argument holds with $g_Z$ and $G_Z$ instead
of $g_Y$ and $G_Y$.

The conditional intensity is a convenient tool to guarantee the 
existence of the process. Indeed, \cite{Haezendonck80}
developed suitable comparison criteria under which
explosion, the situation in which there are 
infinitely many transitions in a finite time span, 
can be prevented. Indeed, their Corollaries~2 and 5 imply, 
for two temporal point processes $X_n$ and $X_n^*$ with 
conditional intensities $\lambda$ and $\lambda^*$, that if
\begin{itemize}
    \item for every $n \in \mathbb{N}_0$, 
       $\lambda_{n+1} \leq \lambda^*_{n+1}$;
    \item for every $n\in\mathbb{N}$, either 
    $\lambda_{n+1}(x; x_1, \dots, x_n)$ or 
    $\lambda^*_{n+1}(x; x_1, \dots, x_n)$
    depends only on $x-x_n$,
\end{itemize}
then the probability of explosion at or before time $x$ of 
the point process defined by $\lambda$ is at most as big 
as that of the point process defined by $\lambda^*$. Under
the same conditions \cite[Corollary~1]{Haezendonck80}, for
all $n\in\N$ and $x\geq 0$,
\[
\mathbb{P}(X_n \leq x ) \leq 
\mathbb{P}(X_{n}^* \leq x).
\]

Below, we
establish existence for two common families of sojourn 
time distributions, the Gamma and the Weibull.

\begin{proposition}
Let $(S_n, X_n)_{n=1}^{\infty}$ be an alternating non-homogeneous
semi-Markov process with values in $\{ 0, 1\} \times 
\mathbb{R}^+$ with  $S_0=1$, $X_0=0$ and semi-Markov
kernels $G_Y(x,\cdot)$, $G_Z(x, \cdot)$ that follow 
Gamma distributions with shape and rate parameters 
$\theta_Y(x) = (k_Y(x), \lambda_Y(x))$ and 
$\theta_Z(x) = (k_Z(x), \lambda_Z(x))$ in 
$[1,\infty) \times (0,\infty)$ such that, 
for all $x\in\mathbb{R}^+$,
\[
\lambda_Y(x) \leq c;
\quad
\lambda_Z(x) \leq c
\]
for some $c > 0$. 
Write $X_\infty = \lim_{n\to \infty} X_n$ for the
time of explosion. Then $\mathbb{P}(X_\infty < \infty)
= 0$.
\label{prop:semimark_gamma}
\end{proposition}

\begin{proof}
The probability density and cumulative distribution 
functions of the Gamma distribution with shape and
rate parameters $k(x)$ and $\lambda(x)$ are, for 
$\tau \geq 0$,
\begin{align*}
 g(x, \tau; k(x), \lambda(x) ) = 
 \frac{ {\lambda(x)}^{k(x)} \tau^{k(x)-1} e^{-\lambda(x) \tau}}{\Gamma(k(x))};
    \qquad 
 G(x, \tau; k(x),\lambda(x)) = 
    \frac{\gamma(k(x), \lambda(x) \tau)}{\Gamma(k(x))},
\end{align*}
writing $\Gamma$ for the gamma function and $\gamma$ for the lower
incomplete gamma function. The conditional intensity is, for $n=0,
1, \dots$ and $0 = x_0 \leq x_1 \leq \dots \leq x_n \leq x$,
\begin{align*}
    \lambda_{n+1}(x; x_1, \dots, x_n) &= 
    \frac{g_T(x_n, x-x_n; k_T({x_n}), \lambda_T({x_n}))}
    {1 - G_T(x_n, x-x_n;k_T({x_n}), \lambda_T({x_n}))} \\
    &= \frac{\lambda_T({x_n})^{k_T({x_n})}
       (x - x_n)^{k_T({x_n})-1}
       e^{-\lambda_T({x_n}) (x - x_n)}}
       {\int_{\lambda_T({x_n}) (x-x_n)}
       ^{\infty}u^{k_T({x_n}) - 1} e^{-u}\,du},
\end{align*}
where $g_T$ is either $g_Y$ or $g_Z$. We examine the limiting 
behaviour as $x\to\infty$. See that
\[
\lim_{x \to \infty}
g_T(x_n, x-x_n;k_T(x_n), \lambda_T(x_n)) = 0, \qquad 
\lim_{x \to \infty} \left(
1-G_T(x_n, x-x_n;k_T(x_n)\lambda_T(x_n)) 
\right) = 0.
\]
Noting that both are differentiable on $(0,\infty)$, by
L'H\^opital's rule, 
\begin{align}
    \lim_{x \to \infty} \lambda_{n+1}(x; x_1, \dots, x_n) &= \lim_{x \to \infty}\frac{\lambda_T(x_n) (x - x_n) - (k_T(x_n) - 1)}
    {x - x_n} = \lambda_T(x_n)
    \label{eq:hazard_rate_gamma}
\end{align}
after simplifying. 

To prove monotonicity, we must show 
that $\lambda_{n+1}(x;x_1, \dots, x_n)$ is increasing in 
$x\geq x_n$. Write $t = \lambda_T(x_n) ( x - x_n)$. Then 
$\lambda_{n+1}(x; x_1, \dots, x_n)$ can be written as
$\lambda_T(x_n) h(t)$ for 
\[
h(t) = \frac{t^{k_T(x_n)-1} e^{-t}}{ 
  \int_t^\infty u^{k_T(x_n) - 1} e^{-u} du}.
\]
Therefore, it suffices to show that the function $t \to \log h(t)$ is 
non-decreasing in $t > 0$. Now,
\begin{align*}
    \frac{\partial}{\partial t} \log h(t) &= 
    \frac{k_T(x_n) -1}{t} - 1 + \frac{t^{k_T(x_n) -1} e^{-t}}{\int_{t} ^{\infty} u^{k_T({x_n}) - 1} e^{-u}\,du}.
\end{align*}
If $t < k_T(x_n) - 1$, we see directly that
the derivative is positive. Otherwise, use 
integration by parts to simplify the last term in the 
right-hand side to
\begin{align*}
    1 -  \int_{t}^{\infty} \frac{k_T(x_n)-1}{u}
     u^{k_T({x_n}) - 1} e^{-u}\,du
   \Big/ \int_{t} ^{\infty} u^{k_T({x_n}) - 1} e^{-u}\,du.
\end{align*}
Consequently
\begin{align*}
    \frac{\partial}{\partial t} \log h(t) &= 
   \int_{t}^{\infty} \left\{
   \frac{k_T(x_n)-1}{t} - \frac{k_T(x_n)-1}{u}
   \right\}
     u^{k_T({x_n}) - 1} e^{-u}\,du
   \Big/ \int_{t} ^{\infty} u^{k_T({x_n}) - 1} e^{-u}\,du 
\end{align*}
is non-negative. We conclude that 
$ \lambda_{n+1}(x; x_1, \dots, x_n) $ 
is bounded by $\lambda_T(x_n)$ for all 
$k_T(x_n) \geq 1$.

Recall that we assume that 
$\sup_{x\in\mathbb{R}^+} \max( \lambda_Y(x),
\lambda_Z(x) ) \leq c$. We may then construct a Poisson
process $N^*$ with conditional intensity
$\lambda^*_{n+1}(x;x_1, \dots, x_n) = c$. 
Clearly, $\lambda^*$ satisfies the second condition of 
\cite[Corollary~2]{Haezendonck80}. Moreover,
$\lambda_{n+1} \leq c = \lambda^*_{n+1}$ for every 
$n \in \mathbb{N}_0$. Since a Poisson process with constant 
intensity has probability zero to explode, we conclude
that $\mathbb{P}(X_\infty < \infty) = 0$.
\end{proof}

Important special cases include $k_T(x) = 1$ for exponential 
distributions or, more generally, $k_T(x) \in \mathbb{N}$ 
corresponding to Erlang distributed phases.

\begin{proposition}
Let $(S_n, X_n)_{n=1}^{\infty}$ be an alternating non-homogeneous
semi-Markov process with values in $\{ 0, 1\} \times 
\mathbb{R}^+$ with  $S_0=1$, $X_0=0$ and semi-Markov
kernels $G_Y(x,\cdot)$, $G_Z(x, \cdot)$ that follow 
Weibull distributions with shape and rate parameters 
$\theta_Y(x) = (k_Y(x), \lambda_Y(x))$ and 
$\theta_Z(x) = (k_Z(x), \lambda_Z(x))$ in 
$(0,\infty) \times (0,\infty)$ such that 
(i) $\lambda_Y(x) \leq c$,
$\lambda_Z(x) \leq c$ for some $c > 0$ and
(ii) either $1 \leq k_Y(x) \leq k$,
$1 \leq k_Z(x) \leq k$ for some $k \geq 1$, or
$k_Y(x) = k_Z(x) = k$ for some $k> 0$.
Write $X_\infty = \lim_{n\to \infty} X_n$ for the
time of explosion. Then $\mathbb{P}(X_\infty < \infty)
= 0$.
\label{prop:semimark_weibull}
\end{proposition}

\begin{proof}
Let $G_T(x, \cdot)$ and corresponding $(\lambda_T(x), k_T(x))$ correspond to either $Y$- or $Z$-phase cases. The probability density and cumulative distribution
functions of the Weibull distribution with shape
and rate parameters $k(x)$ and $\lambda(x)$ are,
for $\tau \geq 0$,
\begin{eqnarray*}
    g(x, \tau; k(x),\lambda({x})) & = &
    k(x) \lambda(x) 
    \left( \lambda({x}) \tau \right)^{k(x)-1} 
    e^{-(\lambda({x}) \tau)^{k(x)}}; \\
    G(x, \tau; k(x)\,\lambda({x})) & = &
    1 - e^{-( \lambda({x}) \tau)^{k(x)}}.
\end{eqnarray*}
The conditional intensity is therefore, for $n=0,
1, \dots$ and $0 = x_0 \leq x_1 \leq \dots \leq x_n \leq x$,
\[
    \lambda_{n+1}(x;x_1, \dots, x_n) = k_T(x_n) \lambda_T(x_n) \left( \lambda_T(x_n) (x-x_n) \right)^{k_T(x_n)-1}.
\]
Since the conditional intensity is unbounded, we cannot use
a Poisson process to bound $\lambda_{n+1}$. Instead
we turn to a homogeneous renewal process $N^*$ with 
sojourn times that are Weibull distributed with
shape parameter $k$ and rate parameter $c$. By the
strong law of large numbers, since the expected 
sojourn times are strictly positive, $N^*$ has 
explosion probability zero 
\cite[Section~3.1]{Ross96}. Also, 
\[
\lambda_{n+1}(x; x_1, \dots, x_n) \leq 
\lambda^*_{n+1}(x; x_1, \dots, x_n) = 
k c^k (x-x_n)^{k-1}
\]
and both conditional intensities are a function of
$x-x_n$ only. By \cite[Corollary~2]{Haezendonck80},
$\mathbb{P}(X_\infty < \infty) = 0$.
\end{proof}

The case that $k = 1$ corresponds to exponential
sojourn times.

\subsection{Renewal function: existence and boundedness}

The process counting the number of cycles 
having occurred by time $t \geq 0$ can be written as
\begin{equation}
    N(t) = \sup\left\{n \in \mathbb{N}_0: 
    X_{2n} \leq t\right\},
    \label{eq:countproc_semimarkov}
\end{equation}
where a cycle is an interval of time within which each 
state occurs once. The distribution of $X_{2n}$, the 
jump time after completing the $n$-th cycle, is, for
$n\in\mathbb{N}_0$ and $t\geq 0$,
\begin{equation*}
    F_{2n}(t) = \mathbb{P}\left(
       \sum_{i=1}^{2n} T_i \leq t \right) = 
       \mathbb{P}(X_{2n} \leq t) =
       \mathbb{P}(N(t) \geq n).
\end{equation*}
The \textit{renewal function} is defined, analogously to that of the 
classic alternating renewal process, as  
$M(t) = \mathbb{E}N(t)$, $t \geq 0$ \cite{Janssen06}. 
In our case,
\begin{align*}
    M(t) &= \mathbb{E}N(t) = 
    \sum_{n=0}^{\infty} \mathbb{P}(N(t) > n) = 
    \sum_{n=1}^{\infty} \mathbb{P}(X_{2n} \leq t) = 
    \sum_{n=1}^{\infty} F_{2n}(t), 
\end{align*}
a $2n$-fold convolution.

The following corollaries to 
Propositions~\ref{prop:semimark_gamma}--\ref{prop:semimark_weibull} 
hold.

\begin{corollary}
Let $(S_n, X_n)_{n=1}^{\infty}$ be as in 
Proposition~\ref{prop:semimark_gamma}.
Then its renewal function $M(t)$ satisfies 
$M(t) \leq ct$, $t\geq 0$.
\label{prop:semimark_renewal_pois}
\end{corollary}

\begin{proof}
Construct a Poisson process $N^*(t)$ with intensity $c$
as in the proof of Proposition~\ref{prop:semimark_gamma}
and write $X^*_n$ for its jump times.
By \cite[Corollary~1]{Haezendonck80}, 
$\mathbb{P}(X_{2n} \leq t)$ is bounded from above
by $\mathbb{P}(X^*_{2n} \leq t)$. 
Therefore,
\[
    \mathbb{E}N(t) = 
    \sum_{n=1}^{\infty} \mathbb{P}(X_{2n} \leq t) 
    \leq 
    \sum_{n=1}^{\infty} \mathbb{P}(X^*_{2n} \leq t) 
    = \mathbb{E}N^*(t) = ct.
\]
\end{proof}

\begin{corollary}
Let $(S_n, X_n)_{n=1}^{\infty}$ be as in 
Proposition~\ref{prop:semimark_weibull}. Then
its renewal function $M(t)$ is finite and bounded from 
above by the expectation $\EE (N^*(t))$
of a renewal process $N^*(t)$ with Weibull distributed 
sojourn times having shape parameter $k$ and rate 
parameter $c$. 
\end{corollary}

\begin{proof}
Construct the renewal process $N^*$ as in the proof 
of Proposition~\ref{prop:semimark_weibull}. 
Then, as in the proof of 
Corollary~\ref{prop:semimark_renewal_pois},
$\mathbb{E}(N(t)) \leq  \mathbb{E}(N^*(t))$. 
Also $\mathbb{E}(N^*(t)) < \infty$
(see \cite{Asmussen03} or \cite[Prop. 3.2.2.]{Ross96}).
\end{proof}

\subsection{Age and excess distributions}

Now that the theoretical groundwork for the censoring 
mechanism has been laid, we proceed by determining the 
joint distribution of age and excess. The age $A(t)$ is 
the time elapsed since the last phase change, and 
$B(t)$, the excess, is the time remaining until the 
next phase change. For all $t$ where the process is in 
state $0$, or the $Z$-phase, we assume that the occurrence 
time can be observed perfectly. Therefore we only consider 
age and excess with respect to state $1$, or the $Y$-phase. 
Obtaining their joint distribution allows us to specify the 
likelihood of intervals based on their starting point 
and length in terms of the semi-Markov kernel $G_Y$. 
See Figure~\ref{fig:semimark_visual} for a 
visualisation of the age and excess functions.

\begin{figure}[H]
    \centering
    \begin{tikzpicture}
        \draw[<->, thick] (-0.5,0) -- (13.5,0);
        \draw[->, thick] (0,-0.15) -- (0,3) node at (0,-0.4) {\footnotesize $0$};
        \node at (-1.25,1) {\footnotesize State $1$ (Out)};
        \node at (-1.25,2) {\footnotesize State $0$ (Home)};
        
        \filldraw [black] (0,1) circle (2pt) node at (0.3,1.25) {\scriptsize $X_0$};; 
        \draw[-, very thick] (0,1) -- (2.25, 1);
        \node at (1.125, 0.7) {\scriptsize $T_1$};
        \filldraw [thick] (2.25,-0.15) -- (2.25, 0.15) node at (2.25,-0.4) {\footnotesize $3$};
        \filldraw [black, fill=none] (2.25,1) circle (1.5pt) ;
        
        \filldraw [black] (2.25,2) circle (2pt) node at (2.55,2.25) {\scriptsize $X_1$}; 
        \draw[-, very thick] (2.25,2) -- (3.75, 2);
        \node at (3, 1.7) {\scriptsize $T_2$};
        \filldraw [thick] (3.75,-0.15) -- (3.75, 0.15) node at (3.75,-0.4) {\footnotesize $5$};
        \filldraw [black, fill=none] (3.75,2) circle (1.5pt);
        
        \draw[dotted, very thick] (3.75,0) -- (3.75, 3);
        
        \filldraw [black, fill] (3.75,1) circle (2pt) node at (4.05,1.25) {\scriptsize $X_2$}; 
        \draw[-, very thick] (3.75,1) -- (6.75, 1);
        \node at (5.25, 0.7) {\scriptsize $T_3$};
        \filldraw [thick] (6.75,-0.15) -- (6.75, 0.15) node at (6.75,-0.4) {\footnotesize $9$};
        \filldraw [black, fill=none] (6.75,1) circle (1.5pt);
        
        \draw[thick] (6, 0.9) -- (6, 1.1) node at (6, 1.25) {\scriptsize $t$};
        \draw[->, very thick] (3.75,1) -- (6, 1);
        \node at (4.875, 1.25) {\scriptsize $A(t)$};
        \draw[<-, very thick] (6,1) -- (6.75, 1);
        \node at (6.375, 0.75) {\scriptsize $B(t)$};
        
        \filldraw [black] (6.75,2) circle (2pt)  node at (7.05,2.25) {\scriptsize $X_3$}; 
        \draw[-, very thick] (6.75,2) -- (7.75, 2);
        \node at (7.25, 1.7) {\scriptsize $T_4$};
        \filldraw [thick] (7.75,-0.15) -- (7.75, 0.15) node at (7.75,-0.4) {\footnotesize $11$};
        \filldraw [black, fill=none] (7.75,2) circle (1.5pt);
        
        \draw[dotted, very thick] (7.75,0) -- (7.75, 3);
        
        \filldraw [black] (7.75,1) circle (2pt) node at (8.05,1.25) {\scriptsize $X_4$}; 
        \draw[-, very thick] (7.75,1) -- (9, 1);
        \node at (8.375, 0.7) {\scriptsize $T_5$};
        \filldraw [thick] (9,-0.15) -- (9, 0.15) node at (9,-0.4) {\footnotesize $12$};
        \filldraw [black, fill=none] (9,1) circle (1.5pt);
        
        \filldraw [black] (9,2) circle (2pt) node at (9.35,2.25) {\scriptsize $X_5$}; 
        \draw[-, very thick] (9,2) -- (12, 2);
        \node at (10.5, 1.7) {\scriptsize $T_6$};
        \filldraw [thick] (12,-0.15) -- (12, 0.15) node at (12,-0.4) {\footnotesize $16$};
        \filldraw [black, fill=none] (12,2) circle (1.5pt);
        
        \draw[dotted, very thick] (12,0) -- (12, 3);
        
    \end{tikzpicture}
    \caption{A visualisation of a semi-Markov process with initial values $S_0 = 1$ and $X_0 = 0$. At the dotted line, one cycle has passed - i.e.\ the process has taken both possible state values. The jump times correspond to a change of state. For a given time $t$ in which the process is in state 
    $1$, a non-zero age $A(t)$ and excess $B(t)$ are recorded.}
    \label{fig:semimark_visual}
\end{figure}
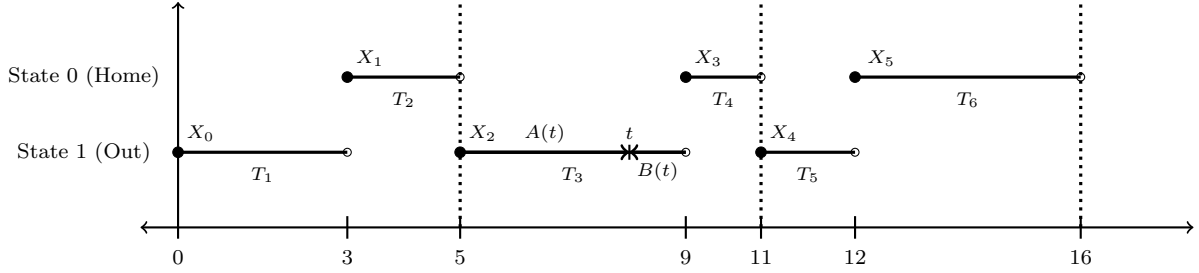

\begin{proposition}
Consider an alternating non-homogeneous semi-Markov 
process $(S_n, X_n)_{n=1}^{\infty}$ with values in 
$\{0,1\} \times \mathbb{R}^+$ with  $S_0=1$, $X_0=0$,
semi-Markov kernels $G_Y$ and $G_Z$ and associated counting
measure $N(t)$, $t\geq 0$. Let the age process with respect 
to the $Y$-phase be 
\[
  A(t) = (t - X_{2N(t)})\,
  \mathbf{1}\{X_{2N(t) + 1} > t\}
\]
and define the excess with respect to the $Y$-phase
as
\[ 
  B(t) = (X_{2N(t) + 1} - t)\,
  \mathbf{1}\{X_{2N(t) + 1} > t\},\]
where $X_{2N(t)}$ is the jump time immediately after
$N(t)$ cycles have been completed. 
Then, for $t\geq 0$ and $0\leq x \leq t$ and $z\geq 0$,
\begin{align}
    \mathbb{P}(A(t) \leq x, B(t) \leq z)
    &= G_Y(0,t) - \int_{t-x}^{t} [1 - G_Y(s, t + z - s)]\,dM(s) - \int_0^{t-x} [1 - G_Y(s, t-s)]\,dM(s) \nonumber \\
    &+ \mathbf{1}\{x=t\}[G_Y(0, t+z) - G_Y(0,t)].
    \label{eq:joint_age_excess_semimark}
\end{align}
\label{prop:age_excess_semimark}
\end{proposition}

\begin{proof}
By construction, $X_0 = 0$ and $S_0 = 1$. 
Now, for $0 \leq x < t$,
\begin{align*}
    \mathbb{P}(A(t) > x) &= \mathbb{P}(t - X_{2N(t)} > x, \,
    X_{2N(t) + 1} > t\,|\,S_0 = 1, X_0 = 0) \\
    &= \sum_{n=0}^{\infty} \mathbb{P}(t - X_{2n} > x,\,
    X_{2n + 1} > t,\,N(t) = n\,|\,S_0 = 1, X_0 = 0) \\
    &= 1 - \mathbb{P}(T_1 \leq t\,|\,S_0 = 1, X_0 = 0) + \sum_{n=1}^{\infty} \mathbb{P}(t - X_{2n} > x,
    \,X_{2n + 1} > t\,|\,S_0 = 1, X_0 = 0)
\end{align*}
after simplifying and removing redundant conditions. 
Note that by (\ref{eq:semi_kernel}) and as we know we are guaranteed to be in state $1$, 
$\mathbb{P}(T_1 \leq t\,|\,S_0 = 1, X_0 = 0) =  
G_Y(0, t)$. Continuing,
\begin{align*}
    \mathbb{P}(A(t) > x) &= 1 - G_Y(0,t) + \sum_{n=1}^{\infty} 
    \mathbb{P}(X_{2n} < t- x;\,X_{2n + 1} > t\,|\,
        S_0 = 1, X_0 = 0) \\
    &= 1 - G_Y(0,t) + \int_0^{t-x} 
    \left[1 - G_Y(s, t-s)\right]\,dM(s),
    \label{eq:age_semimark}
\end{align*}
using the law of total probability and Fubini's theorem. Considering the discrete components,
\begin{equation*}
    \mathbb{P}(A(t) = 0) = 1 - P(A(t) > 0) = G_Y(0,t) - \int_0^{t} \left[1 - G_Y(s, t-s)\right]\,dM(s)
\end{equation*}
and
\begin{align*} 
    \mathbb{P}(A(t) = t) &= \mathbb{P}(X_1 > t) =
    1 - G_Y(0, t).
\end{align*}

Next turn to the excess. For $z \geq 0$, 
\begin{align*}
    \mathbb{P}(B(t) > z) &= \mathbb{P}(X_{2N(t)+1} > t + z\,|\, S_0 = 1, X_0 = 0) \\
    &= 1 - G_Y(0, t+z) + \int_0^{t} \left[1 - G_Y(s, t + z - s)\right]\,dM(s)
\end{align*}
and \begin{equation*}
    \mathbb{P}(B(t) = 0) = 1 - \mathbb{P}(B(t) > 0) = G_Y(0,t) - \int_0^{t} \left[1 - G_Y(s, t-s)\right]\,dM(s).
\end{equation*}
Using similar arguments, we can find an expression 
for the joint probability 
$\mathbb{P}(A(t) > x, B(t) > z)$. For $z \in [0, \infty)$ 
and $x \in [0, t)$, 
\begin{align*}
    \mathbb{P}(A(t) > x, B(t) > z) &= 
    \mathbb{P}(X_{2N(t)+1} > t + z, t - X_{2N(t)} > x\,|\,
       S_0 = 1, X_0 = 0) \\
    &= 1 - G_Y(0, t+z) + \int_0^{t-x} 
      \left[1 - G_Y(s, t + z - s)\right]\,dM(s).
\end{align*}

We can now handle the event $\{ A(t) \leq x, B(t) \leq z \}$ 
for $0 \leq x < t$, $z \geq 0$ as follows:
\begin{align*}
    \mathbb{P}(A(t) \leq x, B(t) \leq z) &= 
    \mathbb{P}(A(t) > x, B(t) > z) + 
    1 - \mathbb{P}(B(t) > z) - \mathbb{P}(A(t) > x) 
    \\
    &= G_Y(0,t) - \int_{t-x}^{t} 
    [1 - G_Y(s, t + z - s)]\,dM(s) - \int_0^{t-x}
    [1 - G_Y(s, t-s)]\,dM(s).
\end{align*}
Finally, for $x=t$, 
\[
\mathbb{P}(A(t) \leq t, B(t) \leq z)
= 1 - \mathbb{P}(B(t) > z) = G_Y(0, t+z) 
- \int_0^{t} \left[1 - G_Y(s, t + z - s)\right]\,dM(s)
\]
and we obtain the proposed expression.
\end{proof}

From Proposition~\ref{prop:age_excess_semimark} we conclude
that the probability that time $t \geq 0$ falls in a 
$Z$-phase is given by 
\begin{align} \label{eq:atom-prob}
    w_t = \mathbb{P}(A(t) \leq 0, B(t) \leq 0) &= G_Y(0,t) - \int_0^{t} [1 - G_Y(s, t-s)]\,dM(s).
\end{align} 
This case constitutes the atomic part of (\ref{eq:joint_age_excess_semimark}). The singular
component on the line $x=t$ has total mass $1 - G_Y(0,t)$
and represents the case that $t$ falls before the first
jump of the semi-Markov process.

The absolutely continuous component of (\ref{eq:joint_age_excess_semimark}) can be written as 
\[
\int_0^x \int_0^z g_Y(t-u, u+v) m(t-u)  du dv
\]
provided that the Radon--Nikodym derivatives $m$ of $M$ 
and $g_Y$ of $G_Y$ exist. Recall that in our proposed 
censoring mechanism, when $t$ falls in a $Y$-phase, the 
entire interval $[t-A(t), t+B(t)]$ is reported, which may 
be parametrised by the left-most point $t-A(t)$ and length
$A(t) + B(t)$. Suppose that $A(t) = u$ and $B(t) = v$, and
apply the change of variables $a = t- u$ and $l = u + v$. 
We find that the joint probability density function of 
left-most point and length is
\begin{align}
    q_t(a, l) = \frac{m(a)g_Y(a, l)}{\int_0^t \left[1 - G_Y(s, t-s)\right]\,dM(s)}
    \mathbf{1} \{ 0 \leq a \leq t \leq a+l; l\geq 0 \},
  \label{prop:density_semimark}  
\end{align}
upon scaling. 

\begin{proposition} \label{prop:forward-sample}
Let $g_Y$ and $m$ be as before, and let $(A,L)$ be distributed according to $q_t(a,l)$ given
by (\ref{prop:density_semimark}). 
Then the marginal probability density function of $A$ at
$a \in [0, t]$ is 
\begin{align}
     f_t(a) & = \frac{m(a)[1-G_Y(a,t-a)]}{\int_{0}^t \left[1 - G_Y(s, t-s)\right]\, dM(s) }
     \label{eq:marginal_dist_sm_a}
\end{align}
and the conditional probability density function of $L$
given $A=a$ is, for $l \in [t-a, \infty)$, 
\begin{align}
    f_{t,\,L\,|A=a}(l) &= \frac{g_Y(a, l)}{1-G_Y(a,t-a)}.
    \label{eq:conditional_dist_sm_a}
\end{align}
\end{proposition}

\begin{proof}
Assume that $0 \leq a \leq t \leq a+l$ and $l \geq 0$. The marginal distribution of the starting time $f_t(a)$ is
\begin{align*}
    f_t(a) &= \int q_t(a,l)\,dl =\frac{m(a)}{\int_{0}^t \left[1 - G_Y(s, t-s)\right]\,dM(s)}\int_{t-a}^{\infty} g_Y(a, l)\,dl \\
    &= \frac{m(a)[1-G_Y(a,t-a)]}{\int_{0}^t \left[1 - G_Y(s, t-s)\right]\,dM(s)},
\end{align*}
and
\begin{align*}
    f_{t,\,L\,|A=a}(l) = \frac{q_t(a,l)}{f_a(a)} = \frac{g_Y(a, l)}{1-G_Y(a,t-a)}.
\end{align*}
\end{proof}
The marginal and conditional distributions of the intervals can be used to generate interval samples. 

\section{A model for non-homogeneous interval-censoring}
\label{sec:complete_model}

\subsection{Model formulation}
\label{sec:mark_dist}

The ensemble of potentially censored occurrence times 
can be mathematically formalised
as a marked point process \cite{Daley03}. The ground process
of points represent the uncensored event occurrences, which
we model by a Markov point process \cite{Lies00}
defined by a probability density with respect to a unit rate
Poisson process. Temporal variations can be 
taken into account as well as interactions between the points. 
Each point is subsequently marked, independently of other points, 
either by an atom at the point when it is observed perfectly,
or by the interval in which the point lies in case of censoring.
The mark kernel that governs the random censoring is based on the 
distribution of age and excess in a non-homogeneous semi-Markov
process.

Formally, let $\mathcal{X}$ be an open set on the real line.
The state space $\mathcal{N}_{\mathcal{X}}$ of a simple 
point process $X$ consists of finite sets 
$\{x_1, x_2, \dots, x_n\} \subset 
\mathcal{X}$, $n \in \mathbb{N}_0$, which we equip with the 
Borel $\sigma$-algebra of the weak topology 
\cite[Appendix~A2]{Daley03}. Let
$p$ be a measurable, non-negative function on $\mathcal{X}$ 
that integrates to unity and $\sim$ a symmetric, reflexive
relation on $\mathcal{X}$.  A point process $X$ on 
$\mathcal{X}$ having probability density $p$ with respect 
to a unit rate Poisson process is Markov 
with respect to $\sim$ if, firstly, $p$ is hereditary, that is,
$p(\mathbf{x})>0$ implies that $p(\mathbf{y})
> 0$ for all subsets $\mathbf{y}$ of $\mathbf{x}$, 
and, secondly, the conditional intensity, defined as
\(
{p(\mathbf{x} \cup \{ t \} )} / {p(\mathbf{x})}
\)
with $a/0 = 0$ for $a\geq 0$,
depends only on the neighbourhood $\{ x \in \mathbf{x} : 
x \sim t \}$ of
$t$ in $\mathbf{x}$ for every $t\in \mathcal{X}\setminus
\mathbf{x}$ and every $\mathbf{x} = \{ x_1, \dots,
x_n\} \subset \mathcal{X}$ for which $p(\mathbf{x})>0$
\cite{Lies00, RipleyKelly77}.

An interaction function is a family $\phi_0, \phi_1, \phi_2,
\dots$ of non-negative functions $\phi_i$ defined on 
configurations of $i$ points that take the value one 
whenever the configuration contains a pair $\{ x_1, x_2 \}$ 
of points that are unrelated, that is, $x_1 \not\sim x_2$.
By the Hammersley--Clifford theorem \cite{RipleyKelly77},
writing $|\cdot|$ for cardinality, a Markov density $p$ can 
be factorised as
\begin{equation} \label{eq:HC}
p(\mathbf{x}) = \prod_{\mathbf{y}\subset\mathbf{x}}
 \phi_{|\mathbf{y}|}(\mathbf{y}) 
\end{equation}
for some interaction function $\phi_i$.
The function $\phi_1(x)$ can be used to model temporal
variations in the likelihood of events occurring. Higher 
order terms $\phi_2, \phi_3, \dots$ govern interactions
between pairs, triples or tuples of points. 

The points $x$ in a realisation $\mathbf{x}$ of $X$ are marked 
independently according to a mark kernel $\nu(\cdot | x)$ on 
$\R\times\R^+$. A mark $(a,l)$ represents an interval 
$[a, a+l]$ that starts at $a$ and has length $l$. The mark 
kernel $\nu$ formalises the semi-Markov censoring discussed in 
Section~\ref{sec:semimark-sub}. For demonstrative purposes, 
we assumed a starting time of $0$, which we now set to 
$-\infty$. Doing so also allows us to ignore the singular
component. Hence the appropriate time-dependent mark kernel
$\nu(\cdot | x)$, $x\in\mathcal{X}$, for a Borel subset
$A \subset \mathbb{R} \times \mathbb{R}^+$ is
\begin{align}
    \nu(A | x) &= 
     \left( 1 - {\int_{-\infty}^x 
      \left[1 - G_Y(s, x-s)\right]\,dM(s)}
    \right) 
    \delta( \{ (x,0) \} \cap A) \nonumber \\
    &+ \int_{-\infty}^x\int_{x-a}^{\infty} 
      \mathbf{1}\{ (a,l) \in A \} \, G_Y(a,dl)\, dM(a).
      \label{def:time_dep_measure:semimark}
\end{align}
Write $W$ for the marked point process defined by $p(\cdot)$ and
$\nu(\cdot | \cdot)$ \cite[Prop.~4.IV]{Daley03}.
A realisation $\textbf{w}$ is of the form
\[
\textbf{w} = \{w_1, w_2, \dots, w_n\} = \{(x_1, (a_1, l_1)), (x_2, (a_2, l_2)), \dots, (x_n, (a_n, l_n))\}
\]
for $a_i \leq x_i \leq a_i+l_i$ for all $i = 1, 2, \dots, n$. 
We denote the set of realisations by $\mathcal{N}_{ {\mathcal{X}}
\times (\mathbb{R} \times \mathbb{R}^+) }$. 

The model description is complete by noting that the 
observable pattern of marks after censoring is
\begin{align*}
    U = \bigcup_{(x_i, (a_i,l_i)) \in W} \{ (a_i, l_i) \}.
\end{align*}
To obtain the probability distribution of $U$, write, for
$F$ in the Borel $\sigma$-algebra of the weak topology on
$\mathcal{N}_{ {\mathcal{\mathbb{R} \times \mathbb{R}^+}}}$,
\[
\mathbb{P}(U\in F | X = \mathbf{x} )  =
\int_{(\mathbb{R}\times\mathbb{R}^+)^n} \mathbf{1}( \{
   (a_1, l_1), \dots, (a_n, l_n) \} \in F ) 
   \prod_{i=1}^n d\nu( (a_i, l_i) | x_i),
\]
where $\mathbf{x} = \{x_1, \dots, x_n\}$, and then take the 
expectation with respect to $X$.

\subsection{Conditional distribution}
\label{sec:conditional_dist}

Write $\mathbf{u}$ for a realisation of the interval set $U$. 
We are interested in the conditional distributions of $X$ 
and $W$ given $U = \uu$.

\begin{theorem}\label{thm:cond_prob_viable_semimark}
Let $W$ be a marked point process with ground process $X$
on the open set $\cX \subset \R$ defined by its 
probability density function $p$ with respect to the 
distribution of a unit rate Poisson process having 
independent marks distributed according to the mark
kernel $\nu( \cdot | x) $  for $x \in \cX$ given by 
(\ref{def:time_dep_measure:semimark}).
Let $\uu$ be a realisation of $U$ that consists of an 
atomic part $\{ (a_1, 0), \dots, (a_m, 0) \}$, $m\in\N_0$, 
and a non-atomic
part $\{ (a_{m+1}, l_{m+1}),\, \dots, (a_n, l_n) \}$, 
$n \geq m$.
Then the conditional distribution of $X$ given $U = \uu$ 
satisfies, for $A$ in the Borel $\sigma$-algebra of the 
weak topology on $\cN_\cX$,
\begin{align*}
    \mathbb{P}(X \in A\,|\,U=\mathbf{u}) = c(\mathbf{u})
    & \int_{\mathcal{X}^{n-m}}
      p(\{a_1, ..., a_m, x_{1}, ..., x_{n-m}\})\,
      1_A(\{a_1, ..., a_m, x_{1}, ..., x_{n-m}\}) \\
    &\sum_{\substack{D_1, ..., D_{n-m} \\
    \cup_j \{ D_j \} = \{1, ..., n-m\}}}
    \prod_{i=1}^{n-m}  1_{[a_{m+i}, a_{m+i} + l_{m+i}]}(x_{D_i})\,dx_i 
\end{align*}
provided that the normalisation constant
\begin{align*}
    c(\mathbf{u}) = 1 / \int_{\mathcal{X}^{n-m}}
    p(\mathbf{x} \cup \{a_1, ..., a_m\}) \,
    \sum_{\substack{D_1, ..., D_{n-m} \\ 
    \cup_j \{ D_j \} = \{1, ..., n-m\}}}
    \prod_{i=1}^{n-m} 
    1_{[a_{m+i}, a_{m+i} + l_{m+i}]}(x_{D_i})\,dx_i
\end{align*}
exists in $(0,\infty)$.
\end{theorem}

\begin{proof}[Proof.]
We must prove, for each $A$ in the Borel $\sigma$-algebra of
$\cN_\cX$ with respect to the weak topology and each $F$ in the 
Borel $\sigma$-algebra of the weak topology on 
$\cN_{\R\times\R^+}$, that
\begin{equation*}
\EE \left[ \1_F(U) \, \PP( X \in A \mid U ) \right] = 
\EE \left[ \1_F(U) \, \1_A(X) \right],
\end{equation*}
as in equation~4 of \cite{LiesMark23}. From the model
description, writing $w_t$ for the modification of 
(\ref{eq:atom-prob}) over $(-\infty, t)$,
$\ell$ for Lebesgue measure, $| \cdot |$ for cardinality,
\begin{align}
    \mathbb{E}[1_F(U)1_A(X)] &=  
    \sum_{n=0}^{\infty}\frac{e^{-\ell(\mathcal{X})}}
    {n!}\int_{\mathcal{X}^n}
       1_A(\mathbf{x}) \, p( \{ x_1, ..., x_n \}) 
    \sum_{C_0 \subset\{1, ..., n\}}
    \frac{1}{(n - |C_0|)!} \prod_{i\in C_0} w_{x_i}  
      \nonumber \\
    & \int_{(\mathbb{R}\times \mathbb{R}_0^+)^{n-|C_0|}}
    1_F(\{ (a_1,l_1),  \dots, (a_{n-|C_0|}, l_{n-|C_0|}) \} \cup
    (\mathbf{x}_{C_0} \times \{{0}\})) \nonumber \\
    &\sum_{\substack{C_1, ..., C_{n-|C_0|} \\ 
    \cup_j \{ C_j \} = \{1, ..., n\} \setminus C_0}}\,\prod_{j=1}^{n-|C_0|} 
    m(a_j) \, g_Y(a_j, l_j) \, 1_{[a_j, a_j + l_j]}(x_{C_j})
    \, da_j dl_j \prod_{i=1}^n dx_i.
    \label{eq:measure_lhs_term}
\end{align}
Following through for the left-hand side,
\begin{align*}
\EE \left[ \1_F(U) \, \PP( X \in A \mid U ) \right]
    & = \sum_{n=0}^{\infty}
    \frac{e^{-\ell(\mathcal{X})}}{n!} 
    \int_{\mathcal{X}^n} p(\{x_1, \dots, x_n\})  
    \sum_{C_0 \subset\{1, ..., n\}}
    \frac{1}{(n - |C_0|)!}  \prod_{i \in C_0} w_{x_i} \\
    & \int_{(\mathbb{R}\times \mathbb{R}_0^+)^{n-|C_0|}} 
    1_F(\{ (a_1, l_1), \dots, (a_{n-|C_0|}, l_{n-|C_0|}) \} \cup
    (\mathbf{x}_{C_0} \times \{{0}\}))  \\
    & \mathbb{P}(X \in A\,|\,U = \{ (a_1, l_1), \dots, (a_{n-|C_0|}, l_{n-|C_0|}) \} \cup
    (\mathbf{x}_{C_0} \times \{{0}\}))  \\
    &\sum_{\substack{C_1, ..., C_{n-|C_0|} \\ 
    \cup_j \{ C_j\} = \{1, ..., n\} \setminus C_0}}
    \prod_{j=1}^{n-|C_0|} 
    m(a_j) \, g_Y(a_j, l_j)\,1_{[a_j, a_j + l_j]}(x_{C_j})
    \, da_j dl_j \,\prod_{i=1}^n dx_i.
\end{align*}
Next, we plug in the expression for 
\(
    \mathbb{P}(X \in A\,|\,U=\mathbf{u}) 
\)
proposed in the statement of the theorem. 
Upon substitution, changing integration order and rearranging,
we obtain
\begin{align*}
 \EE \left[ \1_F(U) \, \PP( X \in A \mid U ) \right]  
  & = \sum_{n=0}^{\infty} \frac{e^{-\ell(\mathcal{X})}}{n!} 
    \sum_{C_0 \subset\{1, ..., n\}} \frac{1}{(n - |C_0|)!}  
    \int_{\mathcal{X}^n} p(\mathbf{y} \cup \mathbf{x}_{C_0})  \, 1_A(\mathbf{y} \cup \mathbf{x}_{C_0})
    \prod_{i \in C_0} w_{x_i} \\
    &\int_{(\mathbb{R}\times \mathbb{R}_0^+)^{n-|C_0|}}  
    1_F(\{ (a_1,l_1),  \dots, (a_{n-|C_0|}, l_{n-|C_0|}) \} \cup
    (\mathbf{x}_{C_0} \times \{{0}\}))  \\ 
    & c( \{ (a_1, l_1), \dots, (a_{n-|C_0|}, l_{n-|C_0|}) \cup
    (\mathbf{x}_{C_0} \times \{ 0 \} ) ) \\
   &   \sum_{\substack{D_1, ..., D_{n-|C_0|} \\ 
    \cup_j \{ D_j \} = \{1, ..., n\} \setminus C_0}}\, 
    \prod_{j=1}^{n-|C_0|} m(a_j) \, g_Y( a_j, l_j) \,
    1_{[a_j, a_j + l_j]}(y_{D_j})  \\
    & \left( \int_{\mathcal{X}^{n-|C_0|}} 
    p(\mathbf{x})
    \sum_{\substack{C_1, ..., C_{n-|C_0|} \\
    \cup_j \{ C_j \} = \{1, ..., n\} \setminus C_0}}\,
    \prod_{j=1}^{n-|C_0|} 1_{[a_j, a_j+l_j]}(x_{C_j}) \,
    \prod_{j\not\in C_0} dx_j \right)
    \prod_{j=1}^{n-|C_0|} da_j dl_j \\
    & \prod_{i\in C_0} dx_i \prod_{k=1}^{n-|C_0|} dy_k 
    = \mathbb{E}[1_F(U)1_A(X)]
\end{align*}
since the term within brackets cancels out against the
normalisation constant $c(\cdot)$.
\end{proof}

Strikingly, although the marking mechanism is more complicated
than that in \cite{LiesMark23}, the conditional distribution of
$X$ has the same form.

The conditional distribution of $W$ can be obtained in the
same vein, by considering $1_{A}(W)$  instead of 
$1_A(X)$ for $A$ a Borel set in
$\mathcal{N}_{\mathcal{X}\times(\mathbb{R}
\times\mathbb{R}^+)}$, the space of marked point
configurations, is given by
\begin{align*}
\mathbb{P}(W \in A\,|\,U=\mathbf{u})
    & \propto \int_{\mathcal{X}^{n-m}}
      p(\{a_1, ..., a_m, x_{1}, ..., x_{n-m}\})\,
      1_A(\{ (a_1, (a_1, 0)), \dots, (a_m, (a_m, 0)), \\
     & ( x_{1}, (a_{m+1},l_{m+1})), \dots , 
   (x_{n-m}, (a_{m+1}, l_{m+1} ) )\}) 
   \prod_{i=1}^{n-m}  1_{[a_{m+i}, a_{m+i} + l_{m+i}]}(x_{i})\,dx_i .
\end{align*}

\section{Modelling considerations}
\label{sec:example}

In this section, we will consider parametric forms for  
$p(\cdot)$ and $\nu(\cdot | x)$.

\subsection{Non-homogeneous point process densities}
\label{sec:non_homogeneous_density}

We will first look at inhomogeneity that manifests itself  
via the occurrence time distribution. In view of (\ref{eq:HC}), 
it is natural to add inhomogeneity by means of the first-order
interaction function $\phi_1$, a procedure known as type~I 
inhomogeneity \cite{VedelJensen01}. The idea is to let 
$\phi_1(\{ x \}) = \beta(x)$ vary over time according to
a measurable function $\beta$ that maps $x \in \mathcal{X}$ to 
$[0,\infty)$. In many applications, it may make sense to model
$\beta$ as a step function. More specifically, given a measurable 
partition $B_k$, $k=1, \dots, K$, of $\mathcal{X}$, set
\begin{align}
    \beta(x) = \sum_{k=1}^{K}\beta_k\mathbf{1}_{B_k}(x), 
    \qquad x \in \mathcal{X} 
    \label{eq:beta_step}
\end{align}
where $\beta_k \geq 0$ is the value that $\beta$ takes in 
the corresponding set $B_k$. 

The function $\phi_1$ can be combined with classic 
second and higher order interaction functions. For instance,
the density of the non-homogeneous area-interaction point 
process \cite{Badd1995} becomes
\begin{align}
    p(\mathbf{x}) = \alpha_p \left( 
    \prod_{x \in \mathbf{x}} \beta(x) 
    \right)
    \exp\left[-\log \gamma\, 
   \ell( \cX \cap U_r(\mathbf{x}) ) \right]
   \label{eq:area_int}
\end{align}
with respect to a unit rate Poisson process on $\mathcal{X}$.
The parameter $\gamma$ quantifies the interaction strength, 
$r$ the radius of interaction, and 
$\alpha_p = c(\beta(\cdot), \gamma)$ is a normalisation 
constant \cite{Badd1995} that depends on the function $\beta$
as well as on $\gamma$. Additionally, $U_r(\mathbf{x}) = 
\bigcup_{i=1}^n B(x_i, r)$ where $B(x_i,r)$ is the closed
interval $[ x_i-r, x_i + r]$. We observe regularity for 
$\gamma < 1$,  clustering for $\gamma > 1$, and $\gamma = 1$ 
corresponds to a non-homogeneous Poisson process with 
intensity function $\beta$. For further examples, we refer
to \cite{Lies00}.

\subsection{Parametric modelling of the mark kernel}
\label{sec:forward}

To proceed, parametric forms for $G_Y$ and $m$ must be
developed. We begin by modelling $G_Y$, the semi-Markov 
kernel that determines the length of time until the 
next transition. We may take one of the time-dependent 
probability density functions considered in 
Section~\ref{sec:hazard_rates}. For instance,
$g_Y(a, l)$ could be the density function of
an exponential distribution with rate 
\begin{align}
    \lambda(a; \alpha) = \alpha \left(b + \sin(c a) \right), 
    \qquad a \in \mathbb{R},
    \label{eq:lambda_func_exp}
\end{align}
where $c$ specifies the period and $b \geq 1$ the elevation 
away from $0$. The parameter $\alpha$ determines the amplitude 
of the harmonic.

We could proceed in a similar fashion for $g_Z$. However,
there are two problems with such an approach. From a probabilistic
point of view, tractable expressions for the renewal density 
$m$ in terms of the semi-Markov kernels $G_Y$ and $G_Z$ do not 
seem to exist, and, statistically speaking, lengths of $Z$ phases 
cannot be observed. Therefore, we shall model $m$ directly. 
The following proposition justifies this approach.

\begin{proposition} \label{prop:exist-renewal}
Let $(S_n, X_n)_{n=1}^\infty$ be a semi-Markov process on $\{ 1 \}
\times \mathbb{R}^+$ with $S_0 = 1$ and $X_0=0$ having semi-Markov
kernel $G_Y$ defined by a density function $g_Y(t, \tau)$, 
$t\in\mathbb{R}^+$, $\tau \in [0,\infty)$ and write $\Tilde m$
for the density of its renewal function
\[
\Tilde M(t) = \sum_{n=1}^\infty \PP( X_n \leq t ).
\]
If $h(t): \mathbb{R^+} \to [0,\infty)$ is a Borel-measurable
function such that $h(t) \leq \Tilde{m}(t)$, then there exists an
alternating semi-Markov process on $\{ 0, 1 \} \times \mathbb{R}^+$
with $G_{01} = G_Y$ and renewal density $h$.
\end{proposition}

\begin{proof}
As $0 \leq {h(t)} / {\Tilde{m}(t)} \leq 1$, we may use a 
time-dependent thinning approach with retention probability 
$p(t) =  {h(t)} / {\Tilde{m}(t)}$. 
Algorithmically, the sought-after process can be constructed
as follows.
Initialise $\Hat{S}_0 = 1$, $\Hat{X}_0 = 0$ and $\Hat{X}_1
= X_1$. Also set 
$\Hat{S}_{2i} = 1$, $\Hat{S}_{2i-1} = 0$ for $i\in\N$
and $j=1$.
For each jump time $X_i$, $i=1, 2, \dots$,
\begin{itemize}
    \item with probability $p( X_{i})$, 
    if $j$ is even, update $\Hat{X}_{j+1} = X_{i+1}$ and 
    increment $j$ by $1$; for odd $j$ update
    $\Hat{X}_{j+1} = \Hat{X}_j$, $\Hat{X}_{j+2} = X_{i+1}$
    and increment $j$ by $2$;
    \item else, 
    if $j$ is odd, update $\Hat{X}_{j+1} = X_{i+1}$ and 
    increment $j$ by $1$; for even $j$ update $\Hat{X}_j =
    X_{i+1}$ leaving $j$ unchanged.
\end{itemize}
Because complete cycles correspond to intervals in between
accepted points
$X_i$, $i=0, 1,2,\dots$,
\[
H(t) = \sum_{n=1}^\infty \PP( \Hat{X}_{2n} \leq t ) = 
\sum_{n=1}^\infty \PP( X_n \leq t; X_n\,{\rm{retained}}) = 
\int_0^t \frac{h(s)}{\Tilde{m}(s)} d\Tilde M(s).
\]
We can therefore conclude that the intensity of the thinned 
process is 
$\frac{h(t)}{\Tilde{m}(t)} \Tilde{m}(t) = h(t)$ (see e.g.\
\cite[pp.~78--79]{Daley03} and hence 
$(\Hat S_n, \Hat X_n)_{n=1}^\infty$ is a non-homogeneous 
alternating semi-Markov process that satisfies the proposed 
conditions.
\end{proof}

As an illustration, suppose that $h$ is a step function
\begin{align}
    m(t) = \sum_{j=1}^{J} \delta_j\mathbf{1}_{A_j}(t), 
 \quad t\in \mathbb{R^+},
 \label{eq:renewal_approximation}
\end{align}
that takes $J$ different values $\delta_j > 0$ on Borel sets $A_j$ 
forming a partition of the half line ($j=1,\dots, J$, with $J\in\N$).
The following corollary lays out conditions under which $h=m$ is the 
renewal density of an alternating renewal process whose $Y$-phases
are governed by (\ref{eq:lambda_func_exp}). 

\begin{corollary}
\label{prop:bound_for_renewalfunc} 
A sufficient condition for (\ref{eq:renewal_approximation}) to be the
renewal density of an alternating semi-Markov process on $\{0,1\} \times\R^+$
with $G_Y$ given by (\ref{eq:lambda_func_exp}) on $\R^+$
is that for all $j=1, \dots, J$ we have $\delta_j \leq \alpha(b-1)$.
\end{corollary}

\begin{proof}
For (\ref{eq:renewal_approximation}) to induce a semi-Markov
process, we require $h(t) \leq \Tilde{m}(t)$, where $\Tilde{m}(t)$ 
is the Radon-Nikodym derivative of the renewal function 
$\Tilde{M}(t)$. In Proposition~\ref{prop:exist-renewal} we defined
\[ \Tilde M(t) = \sum_{n=1}^\infty \PP( X_n \leq t ),\]
with $\left(X_n\right)_{n=0}^{\infty}$ being its associated jump 
process of only $Y$-phases. By construction, its conditional 
intensity is $\Tilde{\lambda}_{n+1}(t; t_1, \dots, t_n) = 
\lambda(t, \alpha)$ for all
$0 \leq t_1\leq \dots \leq t_n \leq t$.

Observe that $\inf \{ \lambda(t; \alpha) : t\in \R \} = 
\alpha(b-1)$. Construct a Poisson process $N^*(t)$
with intensity $\nu = \alpha(b-1)$. By 
\cite[Corollary~1]{Haezendonck80}, since 
$\lambda_{n+1}^*(t; t_1, \dots, t_n) \leq 
\Tilde{\lambda}_{n+1}(t; t_1, \dots, t_n)$, we may conclude 
that the renewal function $\nu t$ of $N^*(t)$ is bounded
from above by $ \Tilde{M}(t)$ for all $t$. Hence also 
$\nu \leq \tilde m(t)$. 

For $h=m$ as in (\ref{eq:renewal_approximation}), in 
order to have 
$\sum_{j=1}^{J} \delta_j\mathbf{1}_{A_j}(t) \leq \alpha(b-1)$, 
it is sufficient that $\delta_j \leq \alpha(b-1)$ for all 
$j = 1, \dots, J$ to guarantee that $m(t)$ is the renewal 
density of a semi-Markov process. 
\end{proof}

As noted before, in practice, the starting point $0$ is moved back 
to $-\infty$. Realisations $\mathbf{u}$ from the specified model 
may be obtained as follows.
First, a set of points $\mathbf{x} \subset \mathcal{X}$ in 
time are chosen according to the probability density 
function $p(\cdot)$ by, for example, coupling from
the past \cite{Kendall98} or the Metropolis--Hastings 
algorithm \cite{Geyer94}. Next, for each point $x 
\in\mathbf{x}$, it is determined whether or not it 
is an atom based on $w_x$. If this 
is not the case, we appeal to 
Proposition~\ref{prop:forward-sample} and 
use rejection sampling with a proposal distribution that
simulates a uniformly distributed point in
$A_j \cap (-\infty, x]$ 
chosen with probability 
\(
\delta_j \ell(A_j\cap(-\infty,x]) / 
   \sum_{i =1}^J \delta_i \ell(A_i \cap(-\infty, x]) 
\)
and acceptance probability 
$\exp[ - \lambda(a; \alpha) (x - a) ]$. The result is a 
sample $a$ from $f_x(a)$, cf.\ (\ref{eq:marginal_dist_sm_a}).
The length
is then sampled according to an exponential distribution
with parameter $\lambda(a; \alpha)$ shifted by $x-a$
(see (\ref{eq:conditional_dist_sm_a})).  It is interesting 
to observe that, in contrast to the 
alternating renewal case studied in \cite{LiesMark23},
using the marginal distribution with respect to $A$ 
and then the conditional given $A$ is computationally 
simpler than sampling $L$ first.

\subsection{Statistical aspects}

In practical applications, both the family of probability 
density functions $g_Y(t, \tau; \theta)$ for the sojourn 
times in phase $Y$ and the function $m(t; \xi)$ rely on
unknown parameters $\eta = (\theta, \xi)$ that must be 
estimated. The log-likelihood $L(\eta; \mathbf{u})$ follows directly from 
(\ref{eq:measure_lhs_term}). Upon observing
$\mathbf{u} = \{ (a_1, 0), \dots, (a_m, 0),$
$(a_{m+1}, l_{m+1}),$ $ \dots, (a_n, l_n) \}$,
\begin{align}
    L(\eta; \mathbf{u}) &= 
    \sum_{i=1}^{m} \log\left(  1 - \int_{-\infty}^{a_i} 
      [1 - G_Y(s, a_i-s; \theta)]\,m(s;\xi) \, ds \right) 
      + \sum_{i=m+1}^n \log\left( 
         m(a_i;\xi) \, g_Y(a_i, l_i;\theta) \right).
    \label{eq:likelihood_semimark}
\end{align}
When the sojourn time distributions $G_Y$ and $G_Z$ and hence 
the renewal density $m \equiv ( \mathbb{E} Y + 
\mathbb{E} Z )^{-1}$ are not time-varying, 
(\ref{eq:likelihood_semimark}) reduces to the renewal 
likelihood in \cite{LiesMark23}. 

We will illustrate the procedure by means of
a specific example. For the sojourn times, we take an
exponential model; for the function $m$, we use 
(\ref{eq:renewal_approximation}).  Assume that 
$G_Y(t, \cdot)$, $t\in \mathbb{R}$, is distributed 
exponentially with rate parameter $\lambda(t)$ as in
(\ref{eq:lambda_func_exp}) and $m$ given by 
(\ref{eq:renewal_approximation}). 
In the homogeneous case that $\lambda(t) \equiv \alpha > 0$
(that is, $b=1$ and $c=0$), $J = 1$, 
$A_1 = \mathbb{R}$ and $0 \leq \delta_1 \leq \alpha$,
\begin{align*}
     L(\alpha, \delta_1; \mathbf{u}) &= 
  m \log \left( 1 - \frac{\delta_1}{\alpha}\right) + 
  (n-m) \log \delta_1 + (n-m) \log \alpha 
  - \alpha \sum_{i=m+1}^n l_i.
\end{align*}
In general, the atom probability for a given time 
$x \in \mathcal{X}$ is
\begin{align}
    w_x = 1 - \sum_{j=1}^{J} \delta_j 
    \int_{A_j \cap (-\infty, x]} 
    e^{-(x-s)\lambda(s)}\,ds.
    \label{eq:atom-prob-exp}
\end{align}
The likelihood equation (\ref{eq:likelihood_semimark})
after substitution and discarding of terms that do not
depend on the parameters becomes
\begin{align*}
    L(\delta, \alpha; \mathbf{u}) 
    &= \sum_{i=1}^{m} \log \left( 
      1 - \sum_{j=1}^{J} \delta_j
      \int_{(a_i - A_j) \cap [0,\infty)]}
      e^{-\alpha r(b +\sin(ca_i-cr))}\,dr
      \right) 
      + \sum_{i=m+1}^n 
    \log\left(\sum_{j=1}^{J} \delta_j\mathbf{1}_{A_j}(a_i)\right) \\
    &+ (n-m) \log \alpha -\alpha \sum_{i=m+1}^{n} 
      l_i \left( b + \sin\left(c a_i\right) \right).
\end{align*}
The resulting  equations can be solved numerically to 
find optimal values for $\delta_k, k = 1, \dots, J$, and 
$\alpha$ under the inequality constraints $0 \leq \delta_j \leq \alpha(b-1)$, $j=1, \dots, J$.

The distribution of unobserved occurrence times may also be 
considered as a parameter to be estimated using the
reported intervals. To do so, since the form of the 
conditional distribution of $W$ given $U$ according to
Theorem~\ref{thm:cond_prob_viable_semimark} is identical
to that for alternating renewal process-based censoring,
the simulation techniques developed in \cite{LiesMark23} 
to obtain realisations of the marked occurrence times given
a sample $\mathbf{u}$ of $U$ apply. Briefly, estimation 
of any parameters involved in $p(\cdot)$, for instance 
the $\beta_k$ in (\ref{eq:beta_step}), requires a 
Monte Carlo EM approach \cite{Geye99}. Once the parameters 
have been estimated, a Metropolis-Hastings algorithm 
\cite{Brooks11, MeynTweed09, Moller04} for a fixed number 
of points can be used. For further details and conditions 
under which the algorithm converges to the desired distribution,
we refer to Propositions~4.3--4.5 in \cite{LiesMark23}.

\section{Illustrations in practice}
\label{sec:simulation}

To show how the non-homogeneous semi-Markov model behaves,
we present a few examples that compare the new model with a 
homogeneous one. Recall that, broadly speaking, there 
are three sources of inhomogeneity: the interval lengths as governed 
by $g_Y$, the renewal density $m$, and the ground process responsible for 
the uncensored event occurrences. Throughout this section, we
set $\cX = (0,1)$.

\subsection{Model mis-specification}
\label{sec:inhomog_misspec}

The first source of inhomogeneity in our  model is the semi-Markov 
kernel $G_Y(a,l)$ for starting point $a\in \mathbb{R}$ and length $l\geq 0$, 
which determines the time until the next transition. For specificity, let us
assume that the actual interval censoring mechanism is governed by a Weibull 
distribution with shape parameter $k=1$ and rate parameter 
$\lambda_Y(t;\alpha) = \alpha(1.6+\sin(2\pi t))$ for $\alpha = 1$. 
Regarding the other model ingredients, we take $p(\cdot)$ of the
form given in (\ref{eq:area_int}) with $\beta = 400$ and $\gamma = 1$, that is, a 
homogeneous Poisson process on $\mathcal{X}$ with intensity $400$. We additionally
set $m(t) = 0.6\, \mathbf{1}_{ [-0.2, 1) }(t)$. 

To illustrate the effect of erroneously assuming a homogeneous 
model, we sample a realisation from the actual model and fit a 
Weibull distribution with parameters $k>0$ and constant 
rate $\lambda_Y(t; \alpha) \equiv \alpha$ for $\alpha > 0$. 
We obtained parameter estimates $\Hat{k} = 0.9$ and $\Hat{\alpha} 
= 2.0$. The graphs of the survival time densities for 
$t=0.6$ for both models are shown in Figure~\ref{fig:survival_sm}. The 
homogeneous model is able to roughly discern the shape of the distribution, but struggles with the scale. Compared 
to the actual model, for $t=0.6$, it  generates more intervals 
shorter than about $0.5$ and fewer of longer length.

\begin{figure}[H]
    \centering
    \includegraphics[width=.8\textwidth]{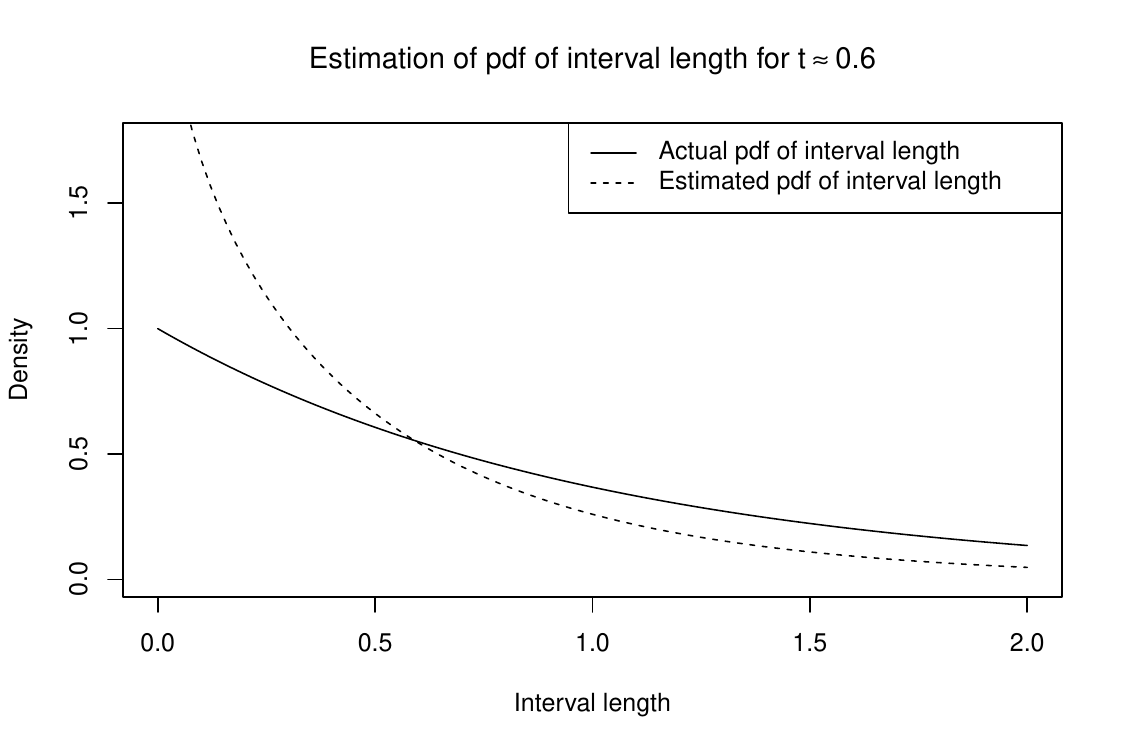}
    \caption{The solid line is the actual probability density of interval length for $k=1$ and $\lambda(0.6;1) = 1$. The 
    broken line is the estimated survival time density.}
    \label{fig:survival_sm}
\end{figure}

\subsection{Inhomogeneity in renewal density and survival time}
\label{sec:inhomog_renewal}

\begin{figure}[thb]
    \centering
    \begin{subfigure}[b]{.45\textwidth}
        \centering
        \includegraphics[width=\textwidth]{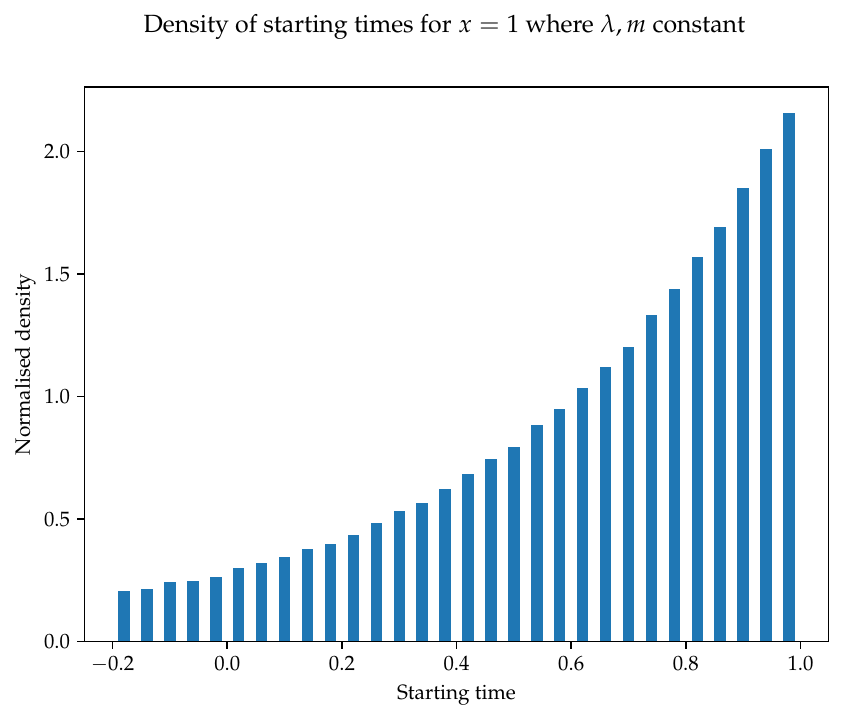}
        \caption{$g_Y$ fixed rate exponential, $m$ constant}
    \end{subfigure}
    \begin{subfigure}[b]{.45\textwidth}
        \centering
        \includegraphics[width=\textwidth]{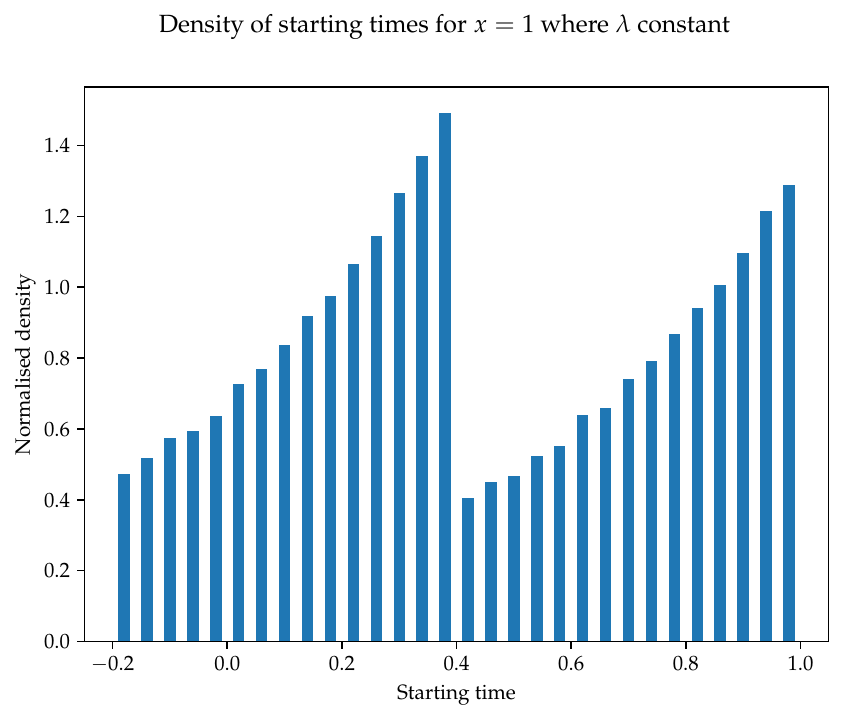}
        \caption{$g_Y$ fixed rate exponential, $m$ varying}
    \end{subfigure}
    \\
    \begin{subfigure}[b]{.45\textwidth}
        \centering
        \includegraphics[width=\textwidth]{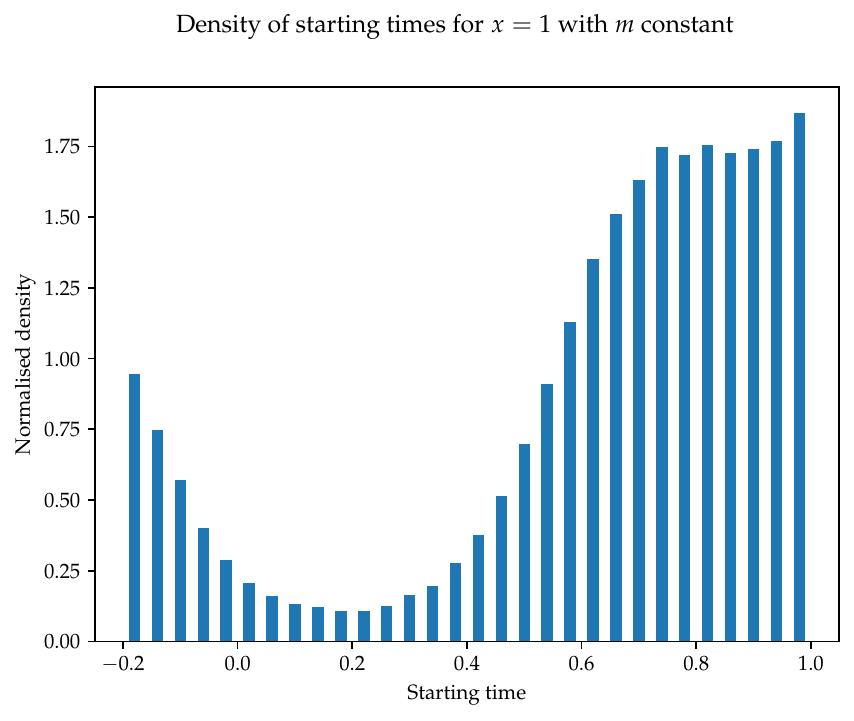}
        \caption{$g_Y$ time-dependent exponential, $m$ constant}
    \end{subfigure}
    \begin{subfigure}[b]{.45\textwidth}
        \centering
        \includegraphics[width=\textwidth]{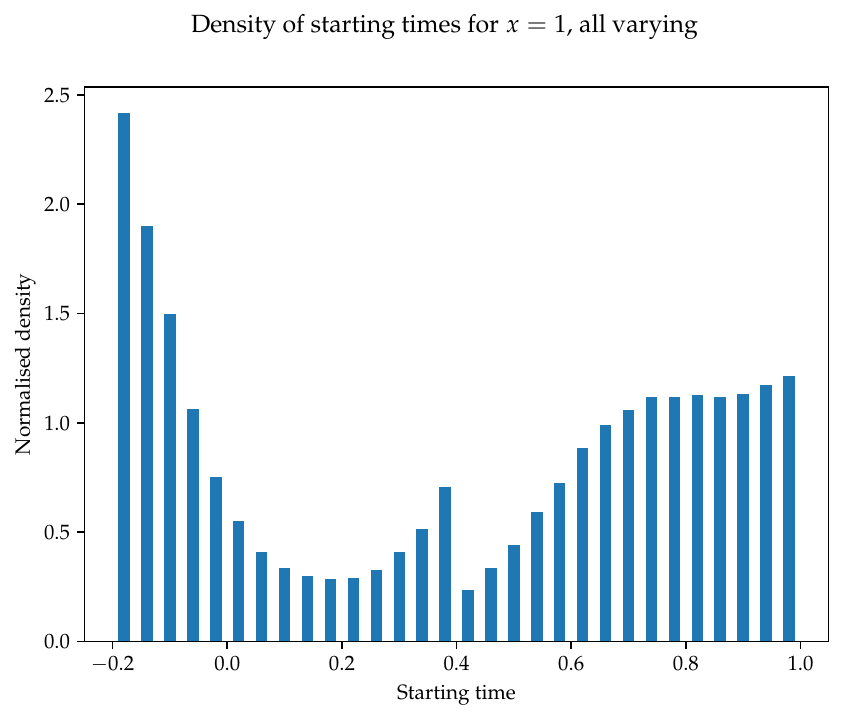}
        \caption{$g_Y$ time-dependent exponential, $m$ varying}
    \end{subfigure}
    \caption{Probability density function of the starting time $f_x(\cdot)$ with $x=1$
    for various choices of $g_Y$ and $m$.}
    \label{fig:renewal_sm}
\end{figure}

In our second experiment, we add inhomogeneity in $m$ to the the model and
study the effect on $f_x$, cf.\ (\ref{eq:marginal_dist_sm_a}). As 
in Section~\ref{sec:inhomog_misspec}, consider an exponential semi-Markov 
density $g_Y$ with rate parameter either constant, $\lambda(t;\alpha) = 1.3\alpha$,
or varying in time according to $\lambda(t; \alpha) = \alpha(1.3 + \sin(2\pi t))$.
Furthermore, set $m(t) = 0.4$ for $t \in [-0.2, 1)$ in the constant case, and 
\[
m(t) = \begin{cases}
            0.4 & t \in [-0.2, 0.4) \\
            0.1 & t \in [0.4, 1)
        \end{cases}
\]
in the time-varying case. We set $\alpha = 1.6$, so the largest
value of $\delta_i$ which guarantees that $g_Y$ is the Radon-Nikodym derivative 
of a semi-Markov kernel is $1.6 \times ( 1.3 - 1 ) = 0.48$.
Figure~\ref{fig:renewal_sm} shows the graphs of $f_x(\cdot)$ for the four 
possible combinations of $g_y$ and $m$ obtained from 
200,000 samples from $q_x$ for $x=1$. In Figures~\ref{fig:renewal_sm}a and 
\ref{fig:renewal_sm}b, we assume that $\lambda$ is constant. When 
$m$ is also constant as in Figure~\ref{fig:renewal_sm}a, the marginal
distribution of the starting times, by Proposition~\ref{prop:forward-sample}, 
is a shifted exponential distribution. If $m$ is allowed to vary in time,
the exponential curve is broken at $t=0.4$, the discontinuity point of $m$,
resulting in a zigzag pattern. In both Figures~\ref{fig:renewal_sm}a and 
\ref{fig:renewal_sm}c, $m$ is constant, but in Figure~\ref{fig:renewal_sm}c 
the rate parameter of $g_Y$ varies according to a harmonic. The resulting
sinusoidal modulation is clearly visible.  Finally, allowing $m$ to vary 
too results in a break at its discontinuity point $t=0.4$ as seen 
in Figure~\ref{fig:renewal_sm}d.

\subsection{Inhomogeneity in occurrence time distribution}
\label{sec:inhomog_occurrencetimes}

In the previous subsections, we have assumed that the first order interaction 
function $\beta$ of the point process $X$ of occurrence times remains 
constant over the entire sampling window $(0,1)$. In our final example, we 
relax this assumption in that we consider a `peak time' in which events are 
more likely to occur and investigate the effect on the conditional 
distribution of occurrences. More precisely, we take an area-interaction
model (\ref{eq:area_int}) with
\[
\beta(y) = \begin{cases}
            3 & y \in (0, c_1) \\
            5 & y \in [c_1, c_2) \\
            3 & y \in [c_2, 1)
        \end{cases} 
        \]
and critical range  $[c_1, c_2) = [0.81, 0.85)$. The radius of interaction
is set to $r=0.1$ and we consider both a regular $(\eta = -1.2)$ 
and a clustered $(\eta = 1.2)$ model. 

As in \cite{LiesMark23}, consider the set $\mathbf{u} = \{(0.45, 0.4), 
(0.51, 0), (0.58, 0)\}$ that contains one non-degenerate interval. Recall 
that the entries are parameterised as $(a,l)$, where $a$ is the starting point 
and $l$ is the length. Figure~\ref{fig:intensity_comparison} plots the 
conditional distribution of the occurrence time on the interval $[0.45, 0.85]$ 
given ${\mathbf{u}}$ for the regular and clustered model.

\begin{figure}[bht]
    \centering
    \begin{subfigure}[b]{.45\textwidth}
        \centering
        \includegraphics[width=\textwidth]{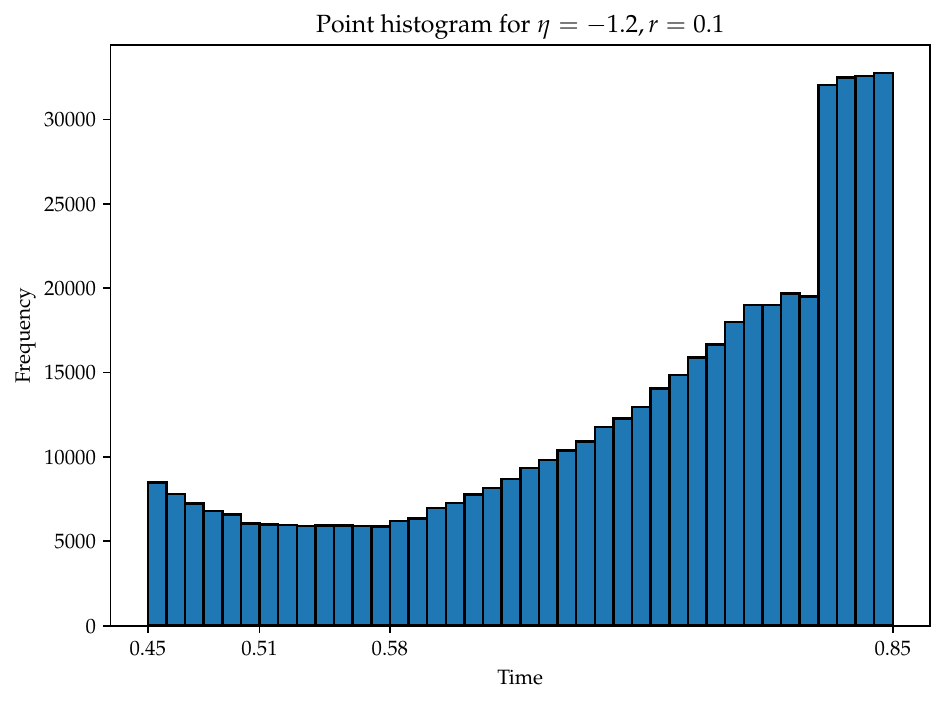}
    \end{subfigure}
    \begin{subfigure}[b]{.45\textwidth}
        \centering
        \includegraphics[width=\textwidth]{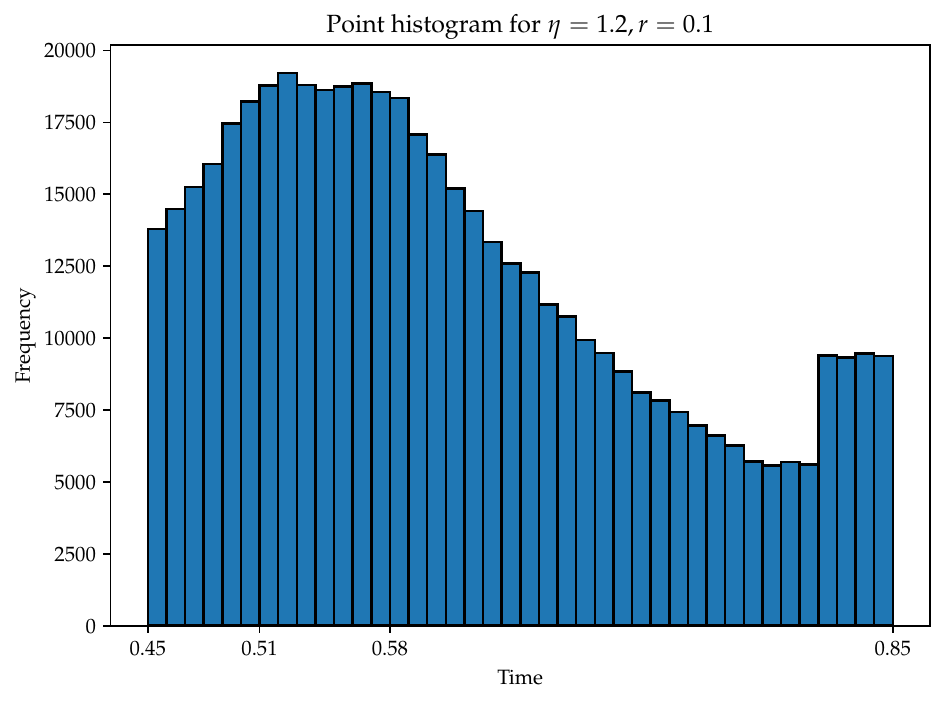}
    \end{subfigure}
    \caption{A comparison between a regular and clustered model 
    with a `peak time' added by changing the intensity function 
    within a critical range.}
    \label{fig:intensity_comparison}
\end{figure}

To create this figure, a Metropolis-Hastings algorithm (see 
Algorithm~4.2, \cite{LiesMark23}) has been run for 600,000 time steps, 
with the first 100,000 iterations being thrown out due to burn-in. 
The general shape of the graphs is similar to the corresponding plots
for constant $\beta = 3$ in Figures~2~and~3 in \cite{LiesMark23}. For the
clustered model, the occurrence time is more likely to happen close to
the atoms, for regular models the probability density is shifted away
from the atoms. In the non-homogeneous case, the higher value of $\beta$ 
during the peak times causes a clear bump in the range $[c_1, c_2) = 
[0.81, 0.85)$. 

\section{Conclusion}
\label{S:conclusion}

We introduced a time-dependent interval censoring mechanism that splits
time into observable and partially observable phases by means of a 
non-homogeneous semi-Markov process on the real line. The process was
shown to be well-defined for a range of Gamma and Weibull semi-Markov kernels.
We extended tools from renewal theory to derive families of time-dependent 
joint distributions of age and excess, which in turn characterise the 
probability distribution of censored intervals. We then constructed a model 
wherein a possibly non-homogeneous point process provides a mechanism 
to select points on the real line, which are independently marked by 
the intervals resulting from the censoring mechanism. For this model, 
a conditional distribution form was posited and verified. The influence
of the model components was demonstrated through parameterised examples.
In future, we intend to apply this model to data on domestic burglaries
and to add a spatial component.

\section*{Acknowledgement}
The research was funded by the Dutch Research Council NWO 
(OCENW.KLEIN.068).

\end{document}